\documentclass{amsart}

\usepackage{amsmath,amssymb,enumerate}
\usepackage{mathtools}

\usepackage{epsfig}
\usepackage{color}

\newtheorem{theorem}{Theorem}[section]

\newtheorem{proposition}[theorem]{Proposition}
\newtheorem{lemma}[theorem]{Lemma}

\newtheorem{definition}[theorem]{Definition}
\newtheorem{remark}[theorem]{Remark}

\newcommand{\occult}[1]{}

\usepackage{soul}

\newcommand\diam{{\operatorname{diam}}}

\newcommand\eps{\epsilon}

\newcommand\NN{{\mathbb N}}

\newcommand\RR{{\mathbb R}}

\newcommand\ZZ{{\mathbb Z}}

\newcommand{\ph}{\varphi}
\newcommand{\phigeo}{\varphi^{\mathrm{geo}}}

\newcommand{\GGG}{\mathcal{G}}

\newcommand{\DDD}{\mathcal{D}}

\newcommand{\SSS}{\mathcal{S}}
\newcommand{\PPP}{\mathcal{P}}
\newcommand{\QQQ}{\mathcal{Q}}

\newcommand{\NNN}{\mathcal{V}}

\newcommand{\Pexp}{P_\mathrm{exp}^\perp}

\begin{document}

\title[Equilibrium states]{Equilibrium states for certain partially hyperbolic attractors}

\begin{abstract} We prove that a class of partially hyperbolic attractors introduced by Castro and Nascimento have unique equilibrium states for natural classes of potentials.  We also show if the attractors are $C^2$ and have invariant stable and centerunstable foliations, then there is a unique equilibrium state for the geometric potential and its 1-parameter family.  We do this by applying general techniques developed by Climenhaga and Thompson.
\end{abstract}

\author{Todd Fisher and Krerley Oliveira}

\address{T.~Fisher, Department of Mathematics, Brigham Young University, Provo, UT 84602, USA, \emph{E-mail address:} \tt{tfisher@mathematics.byu.edu}}
\address{K.~Oliviera, Instituto de Matem\'atica, UFAL 57072-090 Macei\'o, AL, Brazil, \emph{E-mail address:} \tt{krerley@gmail.com}}

\thanks{T.F.\ is supported by Simons Foundation Grant \# 239708.  K.O.\ is partially supported by CNPq, CAPES, FAPEAL, INCTMAT and Foundation Louis D}


\maketitle

\section{Introduction}

The notions of topological pressure and equilibrium states were introduced by  Sinai, Ruelle, and Bowen \cite{Bow75, dR76, jS72}.  These are generalizations of the notion of topological entropy and measures of maximal entropy, and they provide many useful invariants for studying the properties of a dynamical system.  

For a diffeomorphism $f:M\to M$ of a compact manifold and a continuous function $\ph:M\to \mathbb{R}$, called a {\it potential function}, the {\it topological pressure} is  $P(\ph; f)=\sup_{\mu}(h_\mu(f) + \int \ph d\mu)$ where the supremum is taken over all $f$-invariant Borel probability measures.  An $f$-invariant Borel probability measure that maximizes the quantity $h_\mu(f) + \int \ph d\mu$ is an {\it equilibrium state}.

A long standing problem is to find conditions that guarantee the existence and/or uniqueness of equilibrium states. 
For Axiom A diffeomorphisms and H\"older continuous potential functions there is a unique equilibrium state restricted to each basic set \cite{Bow75}.  Outside of the hyperbolic setting there are a number of results where the class of diffeomorphisms and potential functions are restricted \cite{BV, BT09, BCFT, BS03, CT1, DKU, Ho2}.  

Recently, new symbolic tools were obtained by Sarig \cite{Sarig13} for  $C^{1+\alpha}$ diffeomorphisms of surfaces with positive topological entropy. This result allowed Buzzi, Crovisier, and Sarig \cite{BCS} to obtain uniqueness of maximal entropy measures for transitive $C^\infty$ surface  diffeomorphisms with  positive topological entropy. 

For non-uniformly expanding maps there have been a number of results on the existence and uniqueness of equilibrium states \cite{ O2003, OV,  VV}.  Recently, Castro and Nascimento \cite{CastroNascimento} examined measures of maximal entropy for a class of partially hyperbolic attractors that arise from non-uniformly expanding maps.  

We examine equilibrium states for potentials defined on the partially hyperbolic attractors as described in \cite{CastroNascimento}.
Before stating the results we define the class of attractors we will investigate.

We first describe the non-uniformly expanding maps that we will use in the definition of the attractors.
Let $N$ be a connected compact Riemannian manifold and $g:N\to N$ be a topologically exact local diffeomorphism with Lipschitz inverse branches, this means that there exists a function $L:N\to \mathbb{R}^+$ such that for all $x\in N$ there is a neighborhood $U_x$ of $x$ where $g_x:=g|_{U_x}:U_x\to g(U_x)$ is invertible and 
$$
d(g_x^{-1}(y), g_x^{-1}(z))\leq L(x) d(y,z) \textrm{ for all }y,z\in g(U_x).
$$
The degree of $g$  is the number of pre-images of any $x\in N$ by $g$ and is denoted by $\deg(g)$.  We assume there exists a constant $\lambda_u\in (0,1)$ and an open region $\Omega\in N$ such that we have the following:
\begin{itemize}
\item[(H1)] there exists a constant $L>1$ such that $L(x)\leq L$ for all $x\in \Omega$ and $L(x)<\lambda_u$ for $x\notin \Omega$; and
\item[(H2)] there exists a covering $\mathcal{P}$ of $N$ by injective domains of $g$ such that $\Omega$ can be covered by $q<\deg(g)$ elements of $\mathcal{P}$.
\end{itemize}

We now describe the attractors that arise from $g$.
Let $M$ be a compact manifold and $f:M\to M$ a diffeomorphism onto its image such that there is a continuous surjection $\pi:M\to N$ where
$$
\pi\circ f=g\circ \pi.
$$
Given $y\in N$ we set $M_y=\pi^{-1}(y)$.  Therefore, $M=\bigcup_{y\in N}M_y$.  Note that $f(M_y)\subset M_{g(y)}$, each $M_y$ is compact, and there is a maximum diameter for the sets $M_y$.  The next assumption will ensure that the $M_y$ are local stable manifolds for points in the attractor.
\begin{itemize}
\item[(H3)] Assume there exists some $\lambda_s\in (0,1)$ 
such that
$$
d(f(x), f(y))\leq \lambda_s d(x,y)
$$
for all $x,y\in M_z$ and all $z\in N$.
\end{itemize} 
The set $\Lambda=\bigcap_{n=0}^\infty f^n(M)$ is an attractor and $\Lambda$ is compact and $f$-invariant.  We also need the following fact so that the metric on $M$ is related to the metric on $N$.
\begin{itemize}
\item[(H4)]Given $x,y\in M$ we let $\hat{x}=\pi(x)$ and $\hat{y}=\pi(y)$.  Then there exist 
$f$-invariant 
holonomies $h_{\hat{x}, \hat{y}}:M_{\hat{x}}\cap \Lambda\to M_{\hat{y}}\cap \Lambda$  and a constant $C\geq 1$ such that 
$$
\frac{1}{C}[d_N(\hat x, \hat y) + d_M(h_{\hat{x}, \hat{y}}(x), y)]\leq d_M(x,y)\leq
C[d_N(\hat x, \hat y) + d_M(h_{\hat{x}, \hat{y}}(x), y)]
$$
where $d_M$, and $d_N$ are the metrics on $M$ and $N$ respectively.  Furthermore, we assume that the holonomies are invariant for $f$ so that $f(h_{\hat{x}, \hat{y}}(z))=h_{g(\hat{x}), g(\hat{y})}f(z).$
\end{itemize}

The attractor $\Lambda$ can be described as ``solenoid-like" as it can be shown to be topologically conjugate to the natural extension of the system $(N, g)$.  The topological conjugacy can be useful in proving some of the properties we need, but we will work directly with the system $(\Lambda, f)$. 

In \cite{CastroNascimento} it is shown that there is a unique measure of maximal entropy for $\Lambda$.  Furthermore, it is shown that this measure
\begin{itemize}
\item  has exponential decay of correlations for H\"older continuous functions, 
\item satisfies a Central Limit Theorem for any H\"older continuous function, and 
\item this measure varies continuously with respect to the weak$^*$-topology on the space of measures and the $C^1$-topology on the space of maps.
\end{itemize}

Before stating our main results we need to specify some constants.  Fix $\rho>0$ and let $\Omega_\rho=\bigcup_{x\in \Omega} B_\rho(x)$.  Let 
\begin{equation}\label{eq.alpha}
\alpha<\frac{\log \lambda_u}{\log \lambda_u - \log L}.
\end{equation}
We now define the function 
$$
\Psi_{\rho, \alpha}(\varphi)=\Psi(\varphi)=\alpha\sup_{\pi(x)\in\Omega_\rho}\ph(x) + (1-\alpha)(\sup_{\Lambda}\ph) + \log q + \epsilon(\alpha) + m\log L 
$$
where $m=\mathrm{dim}(N)$ and $\epsilon(\alpha)$ is a function such that $\epsilon(\alpha)\to 0$ as $\alpha \to 1$, see Lemma 3.1 of \cite{VV} or \cite{OV} for a description of $\epsilon(\alpha)$.


\begin{theorem}\label{t.unique} Let $f:\Lambda\to \Lambda$ be as described above and $\varphi:M\to \mathbb{R}$ be a H\"older continuous potential such that 
$\Psi(\varphi)< P(\ph; f|_\Lambda)$,
then $(\Lambda, f, \varphi)$ has a unique equilibrium state.
\end{theorem}

Before stating the next theorem we point out that if $L$ is close to 1, then $\Psi(\ph)$ is close to $\sup_{\Omega_\rho} \phi +\log q$.  The above theorem then shows that so long as the potential is not too ``concentrated" in $\Omega$, then there will be a unique equilibrium state.

The next result shows that if the function $\ph$ does not vary too much, then there is a unique equilibrium state.

\begin{theorem}\label{t.variation}
Let $f:\Lambda\to \Lambda$ be as described above and $\varphi:M\to \mathbb{R}$ be a H\"older continuous potential such that 
$$
\sup \ph - \inf \ph < \log \deg(g) - \log q - \epsilon(\alpha) - m\log L,
$$
then $(\Lambda, f, \varphi)$ has a unique equilibrium state.
\end{theorem}


In addition to the previous hypothesis that were assumed in \cite{CastroNascimento} we need some assumptions on the hyperbolic constants, and associated invariant foliations for the next result.

\begin{itemize}
\item[(H5)] There exists an $f$-invariant splitting $E^s\oplus E^{cu}$ for $\Lambda$ and $\lambda_s> L^{-1}$.  Also, there exists an $f$-invariant center-unstable foliation $\mathcal{W}^{cu}$ of $\Lambda$ that is tangent to the center-unstable subbundle $E^{cu}$ in $\Lambda$, and there exists an $f$-invariant stable foliation $\mathcal{W}^s$ tangent to the stable subbundle $E^s$ in $\Lambda$.
\end{itemize}

Let $\varphi^{\mathrm{geo}}(x)=-\log \det(Df|_{E^{cu}(x)})$ where $E^{cu}$ is the center-unstable subspace at $x$.  This is referred to as the {\it geometric potential}.

\begin{theorem}\label{t.srb}
Let $f$ be $C^2$ satisfying properties (H1)-(H5) above and 
we assume that  
\begin{equation}\label{eq.srbrelation}
\log q + \epsilon(\alpha) + m\log L< \min \left\{ \log \deg g, -\sup  \ph^{\mathrm{geo}}\right\},
\end{equation}

then the following hold:
\begin{itemize}
\item $t=1$ is the unique root of $t\mapsto P(t\varphi^{\mathrm{geo}}; f|_\Lambda)$, 
\item there is an $\epsilon>0$ such that $t\varphi^{\mathrm{geo}}$ has a unique equilibrium states $\mu_t$ for each $t\in (-\epsilon, 1+\epsilon)$, and
\item $\mu_1$ is the unique SRB measure for $f$.
\end{itemize}
\end{theorem}

In Section \ref{s.srb} we describe an example  where the hypothesis hold for a class of systems.  This class of systems is given by a DA-type perturbation in a neighborhood of a fixed point for a uniformly expanding endomorphism.


The paper proceeds as follows.  Section \ref{s.background} lists background material as well as the results from \cite{CT} that we need for our result.  Section \ref{s.decomp} defines the decomposition of orbits we need to apply the theorem from \cite{CT} and we prove properties for this decomposition of specification and the Bowen property. Section \ref{s.proofs} contains entropy and pressure estimates need to complete the proofs for Theorems \ref{t.unique} and \ref{t.variation}.  Section \ref{s.srb} contains the proof of Theorem \ref{t.srb} as well as an example of a system satisfying the hypotheses.

\section{Background}\label{s.background}

In this section we review basic properties of partial hyperbolicity.  We also outline the results of Climenhaga and Thompson in \cite{CT}.

\subsection{Partial hyperbolicity}

Let $M$ be a compact manifold.  Recall that a diffeomorphism $f\colon M\to M$ is \emph{(weakly) partially hyperbolic} if there is a $Df$-invariant splitting $TM=E^s\oplus E^c\oplus E^u$, where at least one of $E^s$ or $E^u$ is nontrivial,  and constants $N\in \NN$, $\lambda>1$ such that for every $x\in M$ and every unit vector $v^\sigma\in E^\sigma$ for $\sigma\in \{s, c, u\}$, we have
\begin{enumerate}
\item[(i)] $\lambda \|Df^N_x v^s\|<\|Df^N_x v^c\|<\lambda^{-1}\|Df^N_x v^u\|$, and
\item[(ii)] $\|Df^N_x v^s\|<\lambda^{-1}<\lambda<\|Df^N_x v^u\|$.
\end{enumerate}

A partially hyperbolic diffeomorphism $f$ admits \emph{stable and unstable foliations} $W^s$ and $W^u$, which are $f$-invariant and tangent to $E^s$ and $E^u$, respectively \cite[Theorem 4.8]{yP04}.  In our situation (H5) assures the existence of a stable foliation and the local stable leaves are given by the $M_y$.  There may or may not be foliations tangent to either $E^c$, $E^s\oplus E^c$, or $E^c\oplus E^u$.  When these exist we denote these by $W^c$, $W^{cs}$, and $W^{cu}$ and refer to these as the center, center-stable, and center-unstable foliations respectively.  For $x\in M$, we let $W^{\sigma}(x)$ be the leaf of the foliation $\sigma\in \{s, u, c, cs, cu\}$ containing $x$ when this is defined.  In our situation we know there are $W^{cu}$ leaves in the attractor.

\subsection{Pressure} 
Let $f\colon X\to X$ be a continuous map on a compact metric space.  We identify $X\times \mathbb{N}$ with the space of finite orbit segments by identifying $(x,n)$ with $(x,f(x),\dots,f^{n-1}(x))$.

Given a continuous potential function $\varphi\colon X\to \mathbb{R}$, write
$
S_n\varphi(x) = S_n^f \ph(x) = \sum_{k=0}^{n-1} \varphi(f^kx)
$.
The $n$th Bowen metric associated to $f$ is defined by
\[
d_n(x,y) = \max \{ d(f^kx,f^ky) \,:\, 0\leq k < n\}.
\]
Given $x\in X$, $\eps>0$, and $n\in \NN$, the \emph{Bowen ball of order $n$ with center $x$ and radius $\eps$} is
$
B_n(x,\eps) = \{y\in X \, :\,  d_n(x,y) < \eps\}.
$
A set $E\subset X$ is $(n,\eps)$-separated if $d_n(x,y) \geq \eps$ for all $x,y\in E$.

Given $\mathcal{D}\subset X\times \mathbb{N}$, we interpret $\DDD$ as a \emph{collection of orbit segments}. Write $\mathcal{D}_n = \{x\in X \, :\,  (x,n)\in \mathcal{D}\}$ for the set of initial points of orbits of length $n$ in $\mathcal{D}$.  Then we consider the partition sum
$$
\Lambda^{\mathrm{sep}}_n(\DDD,\ph,\epsilon; f) =\sup
\Big\{ \sum_{x\in E} e^{S_n\ph(x)} \, :\,  E\subset \mathcal{D}_n \text{ is $(n,\epsilon)$-separated} \Big\}.
$$
The \emph{pressure of $\ph$ on $\DDD$ at scale $\eps$} is 
$$
P(\mathcal{D},\ph,\epsilon; f) = \varlimsup_{n\to\infty} \frac 1n \log \Lambda^{\mathrm{sep}}_n(\mathcal{D},\ph,\epsilon),
$$
and the \emph{pressure of $\varphi$ on $\mathcal{D}$} is
$$
P(\mathcal{D},\ph; f) = \lim_{\epsilon\to 0}P(\mathcal{D},\ph,\epsilon).
$$

Given $Z \subset X$, let $P(Z, \varphi, \epsilon; f) := P(Z \times \NN, \varphi, \epsilon; f)$; observe that $P(Z, \varphi; f)$ denotes the usual upper capacity pressure \cite{Pesin}. We often write $P(\ph;f)$ in place of $P(X, \ph;f)$ for the pressure of the whole space.

When $\ph=0$, our definition gives the \emph{entropy of $\mathcal{D}$}:
\begin{equation}\label{eqn:h}
\begin{aligned}
h(\mathcal{D}, \epsilon; f)= h(\mathcal{D}, \epsilon) &:= P(\mathcal{D}, 0, \epsilon) \mbox{ and } h(\mathcal{D})= \lim_{\epsilon\rightarrow 0} h(\mathcal{D}, \epsilon).
\end{aligned}
\end{equation}

Write $\mathcal{M}(f)$ for the set of $f$-invariant Borel probability measures, and $\mathcal{M}_e(f)$ for the set of ergodic measures in $\mathcal{M}(f)$.
The variational principle for pressure \cite[Theorem 10.4.1]{VO} states that 
\[
P(\varphi;f)=\sup_{\mu\in \mathcal{M}(f)}\left\{ h_{\mu}(f) +\int \varphi \,d\mu\right\}
=\sup_{\mu\in \mathcal{M}_e(f)}\left\{ h_{\mu}(f) +\int \varphi \,d\mu\right\}.
\]
A measure achieving the supremum is an \emph{equilibrium state}.

\subsection{Obstructions to expansivity, specification, and regularity}

Bowen showed in \cite{Bow75} that if $(X,f)$ has expansivity and specification, and $\ph$ has a certain regularity property (now called the Bowen property), then there is a unique equilibrium state. We recall definitions and results from \cite{CT}, which show that non-uniform versions of Bowen's hypotheses suffice to prove uniqueness.

Given a homeomorphism $f\colon X\to X$, the \emph{bi-infinite Bowen ball around $x\in X$ of size $\eps>0$} is the set
\[
\Gamma_\eps(x) := \{y\in X \, :\,  d(f^kx,f^ky) < \eps \text{ for all } n\in \ZZ \}.
\]
If there exists $\eps>0$ for which $\Gamma_\eps(x)= \{x\}$ for all $x\in X$, we say $(X, f)$ is \emph{expansive}. 

\begin{definition} \label{almostexpansive}
For $f\colon X\rightarrow X$ the set of non-expansive points at scale $\epsilon$ is  $\mathrm{NE}(\epsilon):=\{ x\in X \, :\,  \Gamma_\eps(x)\neq \{x\}\}$.  An $f$-invariant measure $\mu$ is  almost expansive at scale $\epsilon$ if $\mu(\mathrm{NE}(\epsilon))=0$.  Given a potential $\varphi$, the pressure of obstructions to expansivity at scale $\epsilon$ is
\begin{align*}
\Pexp(\varphi, \epsilon) &=\sup_{\mu\in \mathcal{M}_e(f)}\left\{ h_{\mu}(f) + \int \varphi\, d\mu\, :\, \mu(\mathrm{NE}(\epsilon))>0\right\} \\
&=\sup_{\mu\in \mathcal{M}_e(f)}\left\{ h_{\mu}(f) + \int \varphi\, d\mu\, :\, \mu(\mathrm{NE}(\epsilon))=1\right\}.
\end{align*}
This is monotonic in $\eps$, so we can define a scale-free quantity by
\[
\Pexp(\varphi) = \lim_{\epsilon \to 0} \Pexp(\varphi, \epsilon).
\]
\end{definition}

\begin{definition} 
A collection of orbit segments $\mathcal{G}\subset X\times \mathbb{N}$ has $($W\,$)$-\emph{specification at scale $\epsilon$} if there exists $\tau\in\mathbb{N}$ and $k_0$ such that for every $\{(x_j, n_j)\, :\, 1\leq j\leq k\}\subset \mathcal{G}$ with $n_j>k_0$, there is a point $x$ in
$$\bigcap_{j=1}^k f^{-(m_{j-1}+ \tau)}B_{n_j}(x_j, \epsilon),$$
where $m_{0}=-\tau$ and $m_j = \left(\sum_{i=1}^{j} n_i\right) +(j-1)\tau$ for each $j \geq 1$.
\end{definition}

The above definition says that there is some point $x$ whose trajectory shadows each of the $(x_i,n_i)$ in turn, taking a transition time of exactly $\tau$ iterates between each one.  The numbers $m_j$ for $j\geq 1$ are the time taken for $x$ to shadow $(x_1, n_1)$ up to $(x_j, n_j)$.

\begin{definition} \label{Bowen}
Given $\mathcal{G}\subset X\times \mathbb{N}$, a potential $\varphi$  has the \emph{Bowen property on $\GGG$ at scale $\epsilon$} if
\[
V(\GGG,\ph,\epsilon) := \sup \{ |S_n\varphi (x) - S_n\varphi(y)| : (x,n) \in \GGG, y \in B_n(x, \epsilon) \} <\infty.
\]
We say $\varphi$ has the \emph{Bowen property on $\GGG$} if there exists $\epsilon>0$ so that $\varphi$ has the Bowen property on $\GGG$ at scale $\epsilon$.
\end{definition}

Note that if $\GGG$ has the Bowen property at scale $\eps$, then it has it for all smaller scales. 


\subsection{General results on uniqueness of equilibrium states}

Our main tool for existence and uniqueness of equilibrium states is \cite[Theorem 5.5]{CT}.

\begin{definition}
A \emph{decomposition} for $(X,f)$ consists of three collections $\mathcal{P}, \mathcal{G}, \mathcal{S}\subset X\times (\NN\cup\{0\})$ and three functions $p,g,s\colon X\times \mathbb{N}\to \NN\cup\{0\}$ such that for every $(x,n)\in X\times \NN$, the values $p=p(x,n)$, $g=g(x,n)$, and $s=s(x,n)$ satisfy $n = p+g+s$, and 
\begin{equation}\label{eqn:decomposition}
(x,p)\in \mathcal{P}, \quad (f^p(x), g)\in\mathcal{G}, \quad (f^{p+g}(x), s)\in \mathcal{S}.
\end{equation}
\end{definition}

Note that the symbol $(x,0)$ denotes the empty set, and the functions $p, g, s$ are permitted to take the value zero. 
 
\begin{theorem}[Theorem 5.5 of \cite{CT}]\label{t.generalM}
Let $X$ be a compact metric space and $f\colon X\to X$ a homeomorphism. 
Let $\ph \colon X\to\RR$ be a continuous potential function.
Suppose that $\Pexp(\ph) < P(\ph)$, and that $(X,f)$ admits a decomposition $(\PPP, \GGG, \SSS)$ with the following properties:
\begin{enumerate}
\item $\GGG$ has (W)-specification at any scale;
\item  $\ph$ has the Bowen property on $\GGG$;
\item $P(\PPP \cup \SSS,\ph) < P(\ph)$.
\end{enumerate}
Then there is a unique equilibrium state for $\ph$.
\end{theorem}

\section{Decomposition, specification, the Bowen property, and nonexpansive points}\label{s.decomp}

In this section we establish a number of the properties needed in Theorems \ref{t.unique} -\ref{t.srb}.  We leave most of the entropy and pressure estimates until the next section.

\subsection{Decomposition}

We first define the decomposition we will use on the orbits in order to apply Theorem \ref{t.generalM}.  As the stable direction is uniformly hyperbolic we can allow the set $\PPP=\emptyset$ and the function $p(x,n)=0$ for all $(x,n).$  An orbit segment $(x,n)\in \SSS$ if  
$$\beta(x,n):=\frac 1 n \sum_{i=0}^{n-1}\chi_{\Omega_\rho}(\pi(f^i x))\geq \alpha$$
 where $\alpha$ is given by \eqref{eq.alpha} and $\chi_{\Omega_\rho}$ is the characteristic function for $\Omega_\rho$.  An orbit segment $(x,n)\in \GGG$ if $n=0$ or for all $j\in \{ 1,..., n\}$ we have
 $$
\beta(f^jx, n-j)= \frac 1 j \sum_{i=n-j}^{n-1}\chi_{\Omega_\rho}(\pi(f^{i}x))<\alpha.
 $$

The definition of  the collection of  orbits $\GGG$  are iterates  chosen  so that the center-unstable direction is uniformly contracted by $f^{-1}$ along orbits segments from $f^n(x)$ to $x$. This is related to the notion of hyperbolic times introduced by Alves \cite{Alves}.

\begin{remark}\label{r.concatenation}
Arguing in a similar fashion as in \cite{CFT_Mane, CFT_BV} we may check the following concatenation property of $\GGG$: if $(x,n)$ and $(f^n(x),m)$ are in $\GGG$, then $(x,n+m)\in \GGG$.  
\end{remark}

We now define the decomposition we use for the orbits segments.  For an arbitrary orbit segment $(x,n)$ we let $s\in \{0,1,..., n\}$ be the largest integer such that 
$(x,s)\in \GGG.$  Notice this implies that $(f^s x, n-s)\in \SSS$.  Indeed, suppose that $(f^s x,n-s)\notin \SSS$, we may consider $k\in \{1,..., n-s\}$, such that $k=\min\{i\geq 1;(f^s x,i)\notin \SSS\}.$ Observe that this imply that $(f^s x,k)\in \GGG$ and since $(x,s)\in \GGG$, we have that $(x,s+k)\in \GGG$, contradicting the maximality of $s$. Thus, $(f^s x,n-s)\in \SSS$ and we have a decomposition.



\subsection{Specification}

We now show that the set $\GGG$ described above has the desired specification property.
The proof will proceed first by proving the specification property at any scale $\delta$ for the map $g$.  As the decomposition is defined in terms of the projection map $\pi$ we have a canonically defined decomposition $(\hat \GGG, \hat \SSS)$
 for $g$ from the decomposition $(\GGG, \SSS)$ for $f$ described above.  More specifically, for $(x,n)\in \GGG$ we let $(\hat x, n)\in \hat \GGG$ and $(x,n)\in \SSS$ implies $(\hat x, n)\in \hat \SSS$.  The decomposition for $(\hat x, n)$ is then defined by $(\hat x, s)\in \hat \GGG$ and $(g^s \hat x, n-s)$ if $(x,n)$ decomposes by to $(x, s)$ and $(f^s x, n-s)$.

We fix a constant related to the contraction 
for good orbit segments.
Let 
$$\theta_\alpha=L^\alpha (\lambda_u^{-1})^{1-\alpha}\in (0,1).$$

 
\begin{proposition}\label{specforg} Given $\epsilon$, there exist $\tau=\tau(\epsilon)$  such that if  
$\{(\hat x_j,n_j)\}_{j=0}^\ell \subset \hat \GGG$, then there exists $\tau_j \leq \tau$ for $j=1,\dots,\ell$ and  some $\hat z \in N$ such that

$$
d(g^{m}(\hat x_j),g^{m+r_{j-1}}(\hat z)) \leq \epsilon,
$$ for $0\leq m \leq n_j$, where $r_{0}=0$ and $r_j = \sum_{i=1}^{j} (n_i+\tau_i)$ for each $j \geq 1$.
\end{proposition}

\begin{proof} 
We know there exists some $\delta_0>0$ such that if $\hat x\in N$ is a point where$(\hat x,n)\in \hat \GGG$, then the inverse branch $g^{-j}_{\hat x}$ of $g^{n-j}$ that sends $g^{n}(\hat x)$ to $g^{n-j}(\hat x)$ is a $(\theta_\alpha)^j$-contraction on the ball of radius $\delta_0$, for $1\leq j \leq n$.  To see the existence of such a $\delta_0$ let $\mathcal{P}$ be a finite open cover of $N$ such that each element of $\mathcal{P}$ is an open set on which $g$ has an invertible branch.  Now let $\delta_1$ be a Lebesgue number for the open cover, and choose $\delta_0=\min\{ \delta_1, \rho\}$.


By hypothesis, we know that given $\epsilon>0$, there exists a $\tau=\tau(\epsilon)\in \mathbb{N}$ such that for any 
two points $\hat y_1,\hat y_2\in N$, there exists some $j\leq \tau$  such that $B_\epsilon(\hat y_2)\subset g^j(B_\epsilon(\hat y_1))$. 


Fix $\epsilon>0$ and let $\delta>0$ be less then $\min\{ \epsilon, \delta_0\}$.  Fix $\tau=\tau(\delta)$.  Let $\{(x_j, n_j\}_{j=0}^\ell\subset \hat \GGG$.  Then we know there exists a set of points $X_{\ell-1}\subset B_\delta(g^{n_{\ell-1}}(\hat x_{\ell -1}))$ and $\tau_{\ell}\leq \tau$ such that $g^{\tau_\ell}(X_{\ell-1})$ is the image of $B_\delta(g^{n_\ell}(\hat x_\ell))$ by the inverse branch $g^{-n_\ell}_{g^{n_\ell}(\hat x_\ell)}$.

Similarly, there exists a nonempty set $X_{\ell-2}\subset B_\delta(g^{n_{\ell-2}}(\hat x_{\ell -2}))$ and $\tau_{\ell-1}\leq \tau$ such that $g^{\tau_{\ell-1}}(X_{\ell-1})$ is the image of $X_{\ell-1}$ by the inverse branch $g^{-n_{\ell-1}}_{g^{n_{\ell-1}}(\hat x_{\ell-1})}$.

Continuing inductively we see that $X_0$ is nonempty and we pick $\hat z\subset X_0$.  This proves that $\hat \GGG$ has specification at any scale for $g$.


\end{proof}

Using the above result we now show that $(\Lambda, f)$ has the specification property for sufficiently small scales for the set $\GGG$.

\begin{proposition}\label{specforg} Given $\epsilon$, there exist $\tau=\tau(\epsilon)$  such that if  
$\{( x_j,n_j)\}_{j=0}^\ell \subset \GGG$, then there exists $\tau_j \leq \tau$ for $j=1,\dots,\ell$ and  some $\hat z \in N$ such that

$$
d(f^{m}( x_j),f^{m+r_{j-1}}( z)) \leq \epsilon,
$$ for $0\leq m \leq n_j$, where $r_{0}=0$ and $r_j = \sum_{i=1}^{j} (n_i+\tau_i)$ for each $j \geq 1$.
\end{proposition}

\begin{proof}
From Proposition \ref{specforg} we know that we have specification when we project to $g$.  The problem is the fiber direction.  To handle this we notice by the uniform contraction on the fiber and the fact that the fibers have a finite diameter that given $\epsilon>0$ there exists some $\tau_s(\epsilon)=\tau_s$ such that for each $x\in \Lambda$ we know there exists some $j\leq \tau_s$ such that for each $k\geq j$ we have $M_{\hat f^k(x)}\subset f^{-k}(B_\epsilon(x)\cap M_{\hat x})$. 

We now denote $\hat \tau(\epsilon)$ to be the constant for specification given by the previous result.  We let $\tau(\epsilon)=\max\{ \hat \tau(\epsilon), \tau_s(\epsilon)\}$.  We see that this gives the desired result by combining the result from the previous proposition with the fiber properties listed above.
\end{proof}

\subsection{Bowen property for the partially hyperbolic attractor} 



From property (H4) we know that if $d_M(x,y)<\rho/C$, then $d_N(\hat x, \hat y)<\rho$.  Also, if $\eta>0$ is sufficiently small and 
$\hat x\in \Omega_\rho^c$ and $d_N(\hat x, \hat y)<\eta$ 
then for each preimage $\hat x_1\in g^{-1}(\hat x)$ there exists a unique $\hat y_1\in g^{-1}(\hat y)$ such that $d(\hat x_1, \hat y_1)\leq \lambda_u d(\hat x, \hat y)$.

\begin{lemma}\label{l.productestimates}
If $\eta>0$ is sufficiently small, $(x,n)\in \GGG$, and $y\in B_n(x, \eta)$, then 
$$
d(f^k x, f^ky)\leq C \eta(\theta_\alpha^{n-k} + \lambda_s^k)
$$
for all $0\leq k\leq n$.
\end{lemma}

\begin{proof}
Since $(x,n)\in \GGG$ and $y\in B_n(x,\eta)$ we know for all $j\in \{ 1,..., n\}$ that we have
 $$
 \frac 1 j \sum_{i=n-j}^{n-1}\chi_{\Omega}(\pi(f^{i}(y)))<\alpha.
 $$ 
By the $f$-invariance of the holonomy we have $h_{\widehat{f^kx}, \widehat{f^ky}}=f^k(h_{\hat x, \hat y})$.  So 
$$d_M(h_{\widehat{f^kx}, \widehat{f^ky}}, f^ky)< \lambda_s^k C\eta$$ 
for all $k\in \{0,..., n\}$.   Also, we know that
$$
d_N(\widehat{f^k x}, \widehat{f^k y})\leq \theta_\alpha^{n-k}  C\eta$$
 for all $k\in \{0, ..., n\}$.  
 
\end{proof}

\begin{lemma}\label{l.BowenG}
Any H\"older continuous potential has the Bowen property on $\GGG$ at scale $\eta$ for $\eta$ sufficiently small.
\end{lemma}

\begin{proof}
We know there exists some $K>0$ and $\beta\in (0,1)$ such that 
$$
|\ph (x)-\ph (y)|\leq K d(x,y)^\beta$$
for all $x,y\in \Lambda$.

Given $(x,n)\in \GGG$ and $y\in B_n(x,\eta)$ we see
$$
\begin{array}{llll}
|S_n\ph (x) - S_n \ph(y)| & \leq K\sum_{k=0}^{n-1} d(f^k x, f^k y)^\beta\\
& \leq KC  \eta^\beta \sum_{k=0}^{n-1}(\theta_\alpha^{n-k} + \lambda_s^k)^\beta\\
& \leq 2^{\beta +1} C K \eta^\beta \sum_{j=0}^\infty (\max\{\theta_\alpha, \lambda_s\})^{j\beta}
=: V<\infty.
\end{array}
$$

\end{proof}

%
%
%
%
%

 \subsection{Pressure estimates of the nonexpansive points.}
 
%
%


The diffeomorphism $f:\Lambda \to \Lambda$ satisfies the given property.
\begin{enumerate}\label{propertyE}
\item[[E]] there exist $\epsilon >0$ 
such that 
for $x \in M$, if there exists a sequence $n_k\to\infty$ with $\frac{1}{n_k}\sum_{i=0}^{n_k-1}\chi_{\Omega_\rho}(\pi(f^i(x))) \leq \alpha$, then $\Gamma_\eps(x)=\{x\}$.
\end{enumerate}

Then as in Theorem 3.4 in \cite{CFT_BV} we know the following holds.

\begin{theorem}\label{t.pressureofexpansiveobstructions}
If $f$ satisfies [E], then we have $\Pexp(\ph,\eps) \leq P(\SSS, \ph)$.
\end{theorem}

We note that  \cite[Lemma 5.9]{CFT_Mane} or  \cite[Lemma 6.12]{CFT_BV} can be easily modified to show that $f|_{\Lambda}$ satisfies property [E].


\section{Proof of theorem \ref{t.unique} and theorem \ref{t.variation}}\label{s.proofs}

From the previous section we have a decomposition of the dynamics where the set $\GGG$ has specification and the Bowen property, and the pressure on the set $\SSS$ is an upper bound for the pressure of the obstructions to expansivity.
To conclude the proofs for Theorems \ref{t.unique} and \ref{t.variation}  using Theorem \ref{t.generalM} we need to estimate the pressure on $\SSS$.  First, we find the entropy of $f:\Lambda\to \Lambda$.

\begin{proposition}\label{p.entropyattractor}
$h_\mathrm{top}(f|_\Lambda)=h_\mathrm{top}(g)=\log \deg(g).
$
\end{proposition}

\begin{proof}
We know that $g:N\to N$ has entropy $\log \deg (g)$, see for instance \cite{OV, VV}.  To see that $h_\mathrm{top}(f|_\Lambda)=h_\mathrm{top}(g)=\log \deg(g).
$ we can apply Ledrappier-Walters formula from \cite{LW} as in \cite{RV} to see that the entropy of $f$ restricted to the attractor is the same as the entropy of $g$.  This proves the result since the pre-image of the map $\pi$ is a uniformly contracting and has finite diameter.
\end{proof}


We now estimate the entropy on $\SSS$. Using estimates about  $g$ and the fact that the semiconjugacy between $f$ and $g$  is contracting on each  fiber, we are able to prove the next result.

\begin{proposition}\label{p.entropyestimates} 
$h(\SSS;f) \leq \log q + \epsilon(\alpha) + m\log L.$
\end{proposition}
To prove the proposition above we make use of some results and notation in Section 6 of \cite{VV}. The basic strategy is to make use of the estimates about the number of dynamical balls of $g^l$ that we need to cover $\hat \SSS_{ln}$, for $n$ big enough,  and make use of the contraction in the fibers to obtain a similar estimate for $f^l$.

Let $\mathcal{P}$ be the partition from (H2) and let $\mathcal{P}=\{Q_0,..., Q_{j-1}\}$ where the first $q$ elements cover $\Omega$.
Denote by $\mathcal{P}^{(n)}$ the set of $n$-cylinders of $g$ by elements in $\mathcal{P}$, that is, 
$$
\mathcal{P}^{(n)}=\{Q_{i_0}\cap g^{-1}(Q_{i_1})\cap\cdots\cap g^{-(n-1)}(Q_{i_{n-1}}): Q_{i_k} \in \mathcal{P}, k\in \{0,\dots, n-1\}\}.
$$

Let  
$$
B(n,\alpha)=\{x\in N; \beta(x,n)\geq \alpha\}.
$$

If we denote by $\mathcal{R}^{(n)}$ the set of $n$-cylinders of $g$ by elements in $\mathcal{P}$ that intersects $B(n,\alpha-\delta)$,  for every $0<\delta<\alpha$, there exists $n_0\geq 1$ such that for $n\geq n_0$ the set 
$
\hat \SSS_{n}
$ is covered by the union of elements of $\mathcal{R}^{(n)}$ (see \cite[p. 578]{VV} for a proof). By the estimates of Lemma 3.1 of \cite{VV} or \cite{OV} 
there exists a function $\epsilon(\alpha)$ such that $\epsilon(\alpha)\to 0$ as $\alpha \to 1$ and for $\alpha$ given and sufficiently large $n$ we have 
$$
\# \mathcal{R}^{(n)} \leq e^{(\log q + \epsilon(\alpha))n}.
$$

From Lemma 6.2 in \cite{VV} if $D>0$, then there exists a $C_0>0$ and a sequence of open coverings $\{\hat \QQQ_k\}$ of $N$ such that $\mathrm{diam}(\hat \QQQ_k)\to 0$ as $k\to \infty$ and every set $E\subset N$ such that $\mathrm{diam}(E)\leq D \mathrm{diam}\hat  \QQQ_k$ intersects at most $C_0D^m$ elements of $\hat \QQQ_k$.

Fix $l>0$ and let $\{\QQQ_k\}$ be the lift of the open covers $\{\hat \QQQ_k\}$. 
Denote the set of $n$-cylinders of $g^l$ by elements in $\QQQ_k$ by $\mathcal{C}_{g^l,j} \QQQ_k$.
Let $\NNN_{l,n,k}$ be the set  of cylinders in $\mathcal{C}_{g^l,n} \QQQ_k$ that intersect any element of $\mathcal{R}^{(ln)}$. From Claim 2 of \cite{VV} we have the following.  

\begin{lemma}\label{l.NNN}  Let $k\geq 1$ be large and fixed. Then, there exists some $n_0$ such that 
$$
\# \NNN_{l,n,k} \leq \# \QQQ_k \times [C_0L^{lm}]^n \times e^{(\log q + \epsilon(\alpha))nl},
$$ for every large $n\geq n_0$. 

\end{lemma}

Furthermore, from Claim 1 in \cite{VV} and $n_0$ perhaps larger we know that for each $0<\delta< \gamma$ and $n\geq n_0$ we have
$
B(k, \gamma)\subset B(ln, \gamma-\delta)$
for each $ln\leq k< (n+1)l.$

Now, we are able to prove Proposition~\ref{p.entropyestimates}. 

\begin{proof}[Proof of Proposition~\ref{p.entropyestimates}]  Fix $\gamma>0$ and $l$ sufficiently large. Given $n\geq 1$, denote by $E_{n} \subset \hat\SSS_{ln}$  any maximal $(n,\gamma)$-separated set for $g^l$. If $\diam \hat \QQQ_k<\gamma$, by the construction,   we have $\NNN_{l,n,k}$ covers $E_n$ and each element of $\NNN_{l,n,k}$ intersects $E_n$ in at most one point. By Lemma~\ref{l.NNN}, we have 
\begin{equation}\label{eq.entro}
\# E_n \leq \# \NNN_{l,n,k} \leq \# \QQQ_k \times [C_0L^{lm}]^n \times e^{(\log q + \epsilon(\alpha))nl}.
\end{equation}

Since the diameter of each fiber $M_x$  is uniformly bounded from above, there exists a natural number $m_0$ such that for each point $ x\in E_j$, we may choose points $F_x=\{y_1(x),y_2,\dots,y_{m_0}(x)\}\subset M_x$ such that $F_x$ is a $\gamma$-dense set in the fiber $M_x$. We claim that 
$$
F_n:=\bigcup\limits_{x\in E_n} F_x,
$$ is a maximal $(n, K\gamma)$-separated set for $f^l$ and $\SSS_{ln}$, where $K=2C$ and $C$ is as in (H4). First, since $(x,ln)\in \hat\SSS$, we know $(y_i(x),ln)\in \SSS$ for every $i=1,\dots, m_0$. To check that $F_n$ is maximal, take any $(z,ln) \in \SSS$. By definition, $\pi(z) \in \hat \SSS_{ln}$ and by the maximality of $E_n$ there exists a point $x\in E_n$ such that for every $i=0,\dots,n-1$ we have 
$$
d_N(g^{li}(x),g^{li}(\pi(z))) \leq  \gamma.
$$  

Consider the holonomy $h_{x,\pi(z)}:M_z\cap \Lambda \rightarrow M_x\cap \Lambda$. By the $\gamma$-density of $F_x$ in $M_x$ we have that there exists a point $y\in M_x$ such that if $z'= h_{x,\pi(z)}(z)$, then
$$
d_M(y,z')\leq \gamma.
$$
Now, we estimate $d_M(f^{li}(y),f^{li}(z))$, for $i=0,\dots,n-1$:
$$
d_M(f^{li}(y),f^{li}(z))\leq d_M(f^{li}(y),f^{li}(z')) + d_M(f^{li}(z'),f^{li}(z)) \leq
$$
$$ 
\lambda^{li} d_M(y,z') + C d_N(g^{li}(x),g^{li}(\pi(z))) \leq K \gamma.
$$

Thus, $F_n$ is a maximal $(n, K\gamma)$-separated set for $f^l$ and $\SSS_{ln}$. In order to estimate $\# F_n$, observe that  $\#F_n\leq \# E_n \times m_0$ and  
$$ 
h(\SSS,K\gamma;f^l) \leq \limsup \frac{1}{n}\log \#F_n \leq lm \log L + (\log q + \epsilon(\alpha))l   + \log C_0
$$
by Equation\eqref{eq.entro}.
Taking the limit when $\gamma \rightarrow 0$:

$$
h(\SSS;f^l) \leq  lm \log L + (\log q + \epsilon(\alpha))l + \log C_0.
$$
Using that $h(\SSS;f)= (1/l)h(\SSS;f^l)$ we finish the proof. 

\end{proof}

\begin{proof}[Proof of Theorem \ref{t.unique}]
The estimates in Theorem \ref{p.entropyestimates} are the key to the pressure estimates of the proofs of Theorem \ref{t.unique} and Theorem \ref{t.variation}.  Just as in the proof of Theorem 3.3 in \cite{CFT_BV} we know that 
$$
\begin{array}{llll}
P(\SSS, \ph; g) &\leq \alpha\sup_{\pi(x)\in\Omega_\rho}\ph(x) + (1-\alpha)(\sup_{\Lambda}\ph) + h(\SSS)\\
& \leq \alpha\sup_{\pi(x)\in\Omega_\rho}\ph(x) + (1-\alpha)(\sup_{\Lambda}\ph) + \log q + \epsilon(\alpha) + m\log L \\
&=\Psi(\ph)< P(\ph; f|_\Lambda).
\end{array}
$$

Hence, Theorem \ref{t.generalM} shows there exists a unique equilibrium state.
\end{proof}

\begin{proof}[Proof of Theorem \ref{t.variation}]
Now assume that $\ph$ is H\"older continuous and that 
$$
\sup \ph - \inf \ph < \log \deg(g) - \log q - \epsilon(\alpha) - m\log L,
$$
then
$\sup \ph-\inf\ph < h_{\mathrm{top}}(f|_\Lambda)-h(\SSS)$.  Then
$
h(\SSS) + \sup \ph < h_{\mathrm{top}}(f|_\Lambda) + \inf \ph$ and this implies that
$
P(\SSS, \ph; f)< P(\ph; f)$.
In the previous section we proved that the pressure of obstructions to expansivity are bounded by the pressure on $\SSS$, then $\ph$ has the Bowen property, and $\GGG$ has specification.  Hence, Theorem \ref{t.generalM} shows there exists a unique equilibrium state.
\end{proof}

\section{Proof of Theorem \ref{t.srb}}\label{s.srb}

In this section we prove Theorem \ref{t.srb} and provide an example
 where the result will hold.  The proof is similar to Theorem C in \cite{CFT_BV} except we do not need to worry about the scale of the perturbation.

From the proof of Lemma 7.1 in \cite{CFT_BV} we see that $P(\ph^{\mathrm{geo}}; f)\geq 0$.  From the fact that 
$$\log q + \epsilon(\alpha) + m\log L<  -\sup  \ph^{\mathrm{geo}}
$$
we see that 
$$
\begin{array}{llll}
\Psi(\ph^{\mathrm{geo}})&=\alpha\sup_{\pi(x)\in\Omega_\rho}\ph^{\mathrm{geo}}(x) + (1-\alpha)(\sup_{\Lambda}\ph^{\mathrm{geo}}) + \log q + \epsilon(\alpha) + m\log L \\
& \leq \sup  \ph^{\mathrm{geo}} + \log q + \epsilon(\alpha) + m\log L<  0 \leq P(\ph^{\mathrm{geo}}; f).
\end{array}
$$
Therefore, by Theorem \ref{t.unique} we know that $\ph^{\mathrm{geo}}$ has a unique equilibrium state.

To show that $P(t\ph^{\mathrm{geo}}; f)$ has a unique equilibrium state for $t\in [0,1]$ we show that $\Psi(t\ph^{\mathrm{geo}})<P(t\ph^{\mathrm{geo}}; f)$ for all $t\in [0,1]$.  Since the inequality is strict it will hold in a neighborhood of $[0,1]$.

Notice that 
\begin{equation}\label{eq.boundedabove}
P(t\ph^{\mathrm{geo}}; f)\geq h_{\mathrm{top}}(f|_\Lambda) + t\inf \ph^{\mathrm{geo}}=\log (\deg g) + t\inf \ph^{\mathrm{geo}} = l_1(t).
\end{equation}
Also, we have 
\begin{equation}\label{eq.boundedbelow}
\Psi(t\ph^{\mathrm{geo}})\leq t\sup \ph^{\mathrm{geo}} + \log q +\epsilon(\alpha) + m\log L = l_2(t).
\end{equation}

Let
$$t_0=-\left(\frac{\log q + \epsilon(\alpha) + m\log L}{\sup \ph^{\mathrm{geo}}}\right).
$$
We know that $t_0\in (0,1)$.  For $t\in (t_0, 1]$ we have 
$\Psi(t\ph^{\mathrm{geo}})<0\leq P(\ph^{\mathrm{geo}}; f|_\Lambda)\leq P(t\ph^{\mathrm{geo}}; f|_\Lambda).$

For $t=0$ we see that $l_1(0)>l_2(0)$ by \eqref{eq.srbrelation}.  We know that the root of $l_2$ is $t_0$.  The root of $l_1$ is $-\log(\deg g)/\inf \ph^{\mathrm{geo}}.$  By \eqref{eq.srbrelation} we know that $t_0<-\log(\deg g)/\inf \ph^{\mathrm{geo}}$ and so $l_1(t_0)>l_2(t_0)$.  This implies that $P(t\ph^{\mathrm{geo}}; f|_\Lambda)\geq l_1(t)>l_2(t)\geq \Psi(t\ph^{\mathrm{geo}})$ for all $t\in [0, t_0]$.

Now assume that $f$ is a $C^2$ diffeomorphism and let $\mathcal{M}_e(f)$ be the collection of $f$-invariant ergodic measures.  For $\mu \in \mathcal{M}_e(f)$, let $\lambda_1 < \cdots < \lambda_\ell$ be the Lyapunov exponents of $\mu$, and let $d_i$ be the multiplicity of $\lambda_i$, so that $d_i = \dim E_i$, where for a Lyapunov regular point $x$ for $\mu$ we have 
\[
E_i(x) = \{0\} \cup \{ v\in T_xM \, :\,  \lim_{n\to\pm \infty} \tfrac 1n \log \|Df^n_x(v)\| = \lambda_i \} \subset T_x M.
\]
Let $k=k(\mu) = \max \{1\leq i\leq \ell(\mu) \, :\,  \lambda_i \leq 0\}$, and let $\lambda^+(\mu) = \sum_{i>k} d_i(\mu) \lambda_i(\mu)$ be the sum of the positive Lyapunov exponents, counted with multiplicity.

The Margulis--Ruelle inequality  \cite[Theorem 10.2.1]{BP07} gives $h_\mu(f) \leq \lambda^+(\mu)$.   Ledrappier and Young \cite{LY} proved that equality holds if and only if $\mu$ has absolutely continuous conditionals on unstable manifolds. This implies that for any ergodic invariant measure $\mu$, we have
\begin{equation}\label{eqn:nonpos}
h_\mu(f) - \lambda^+(\mu)\leq 0,
\end{equation}
with equality if and only if $\mu$ is absolutely continuous on unstable manifolds.  So an ergodic measure $\mu$ is an SRB measure if and only if it is hyperbolic and equality holds in \eqref{eqn:nonpos}.

We now show that $P(\phigeo; f|_\Lambda) \leq 0$.  Combining this with the arguments above we see that $P(\phigeo ; f|_\Lambda) =  0$. 
We know that $\phigeo$ has a unique equilibrium state $\mu$; to show that 
$\mu$ is the SRB measure, we need to show that  $-\lambda^+(\mu) = \int\phigeo\,d\mu$.

Since the splitting on $\Lambda$ is partially hyperbolic and the stable direction is uniformly contracting we know that $\int\phigeo d\mu \geq - \lambda^+(\mu)$ with equality if and only if all Lyapunov exponents associated with $E^{cu}$ are nonnegative.  From this we know that $h_\mu(f) -\lambda^+(\mu)\leq h_\mu(f) + \int \phigeo d\mu$.

Now suppose that there is a nonpositive exponent associated with $E^{cu}$.  Then there exists some set $Z\subset \Lambda$ such that $\mu(Z)=1$ and $v\in E^{cu}(z)$ such that $\lim_{n\to \infty} \frac 1 n \log \|Df^n_z(v)\|\leq 0$.  For $z\in Z$ if we define 
$$
\Lambda^-=\{ x\, : \exists K(x), \frac 1 n \sum_{i=0}^{n-1} \chi_{\Omega}(\pi f^i x)\geq \alpha, \forall n\geq K(x)\},
$$
 then we have $z\in \Lambda^-$.  To see this notice that if $z\notin \Lambda^-$ then by property [E] that there exists a sequence $n_k$ such that $\frac{1}{n_k}\sum_{i=0}^{n_k-1}\chi_{\Omega}(\pi(f^{i}x))<\alpha$ and so 
$$\lim_{n_k\to \infty} \frac{1}{n_k}\log \|Df_z^{n_k}\|<0.$$
Then $\mu(\Lambda^-)=1$ and this implies that 
$$
h_\mu(f) - \lambda^+(\mu)\leq h_\mu(f) + \int \phigeo d\mu\leq P(\SSS, \phigeo)\leq \Psi(\phigeo)<0.$$  
So this implies that 
$P(\phigeo; f|_\Lambda)\leq 0$ and so $P(\phigeo; f|_\Lambda)=0$.

By \eqref{eq.srbrelation} we know that $\sup \phigeo<0$ and so $t\mapsto P(t\phigeo; f|_\Lambda)$ is a convex strictly decreasing function from $\mathbb{R}\to \mathbb{R}$ and hence $1$ is the unique root.

Hence, $h_\mu(f) - \lambda^+(\mu)=0$ and $\mu$ is an SRB measure.  Assume there is a different SRB measure $\nu$ that is ergodic.  Then $h_\nu(f) - \lambda^+(\nu)\leq h_\nu(f) + \int \phigeo d\nu< P(\phigeo; f|_\Lambda)=0$ since $\mu$ is a unique equilibrium state.   This completes the proof of Theorem \ref{t.srb}.

\medskip

\noindent{\it Example.}  We now provide an example of systems that satisfy the conditions of theorem \ref{t.srb}.  Let $g_0$ be a linear expanding map of the $d$-torus with $d$ distinct positive eigenvalues.  For a fixed point $p$ of $g_0$ we deform $g_0$ by a pitchfork bifurcation in a small neighborhood of $p$ so that after the deformation $p$ is a hyperbolic saddle for the local diffeomorphism $g$.  Furthermore, we assume that the perturbation is performed  so that $p$ is slightly econtracting in the direction associated with the weakest eigenvalue and expanding in the other directions, see Figure  \ref{fig:pitchfork}.  By construction we know that there are two new fixed points created for $g$, and $g$ agrees with $f_0$ outside of the neighborhood of $p$.

For the neighborhood of $p$ sufficiently small and the contraction in the weak direction close to 1 we see that $g$ satisfies the properties (H1) and (H2).

\begin{figure}[htbp]
\includegraphics[width=\textwidth]{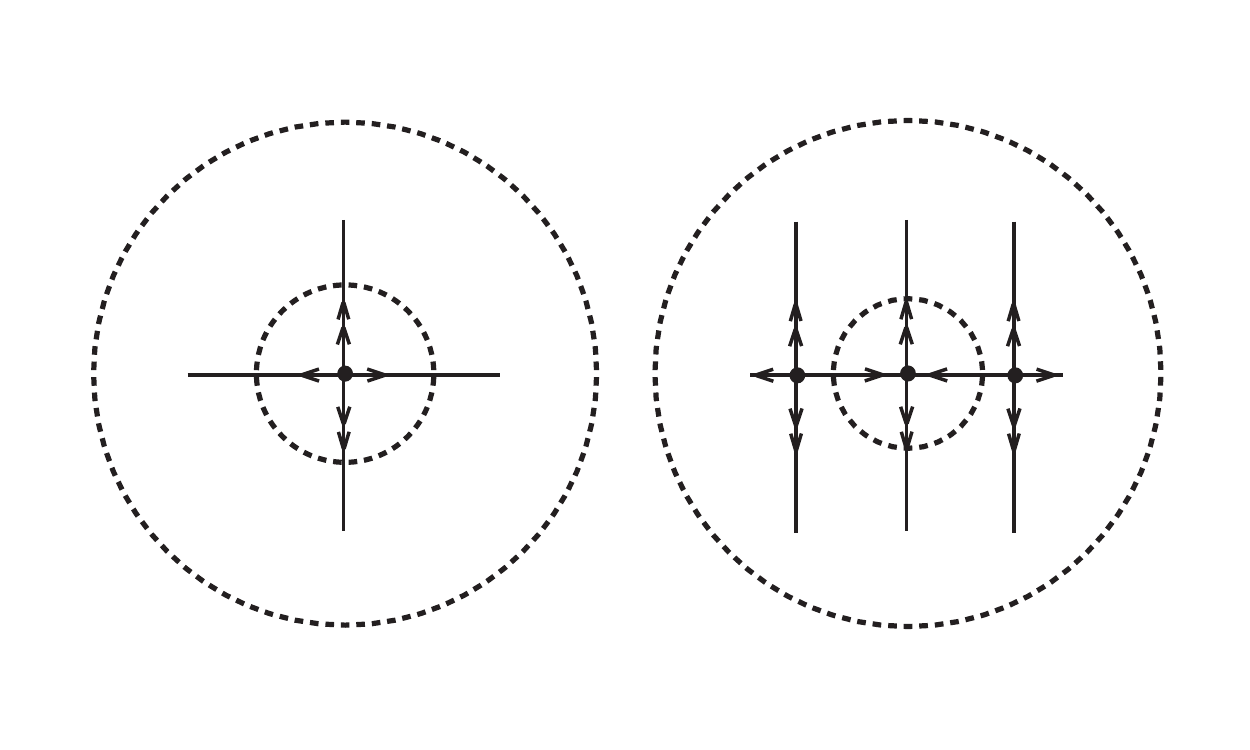}
\caption{Pitchfork bifurcation}
\label{fig:pitchfork}
\end{figure}

We now construct the map $f$.  Let $M=\mathbb{T}^d\times (D^2)^d$ where $D^2$ is the 2-dimensional disk.  Now let 
$$f(t_1,..., t_d, z_1,..., z_d) = (g(t_1,..., t_d)), \frac{1}{2^{\deg g}} z_1 + \frac 1 2 e^{2\pi t_1 i}, \cdots ,\frac{1}{2^{\deg g}} z_d + \frac 1 2 e^{2\pi t_d i} ).$$
One can check that this satisfies the properties (H3)-(H5).  

Furthermore, by the fact that the neighborhood around $p$ can be chosen arbitrarily small and that the contraction at $p$ can be made arbitrarily close to 1 we see that restricted to the attractor $\Lambda$ associated with $f$ we see that the map $f$ and the potential function $\phigeo$ satisfies the hypothesis of Theorem \ref{t.srb}.

%
%


\bibliography{EqStatesPH}{}
\bibliographystyle{plain}
\end{document}